\documentclass[12pt]{article}
\usepackage[utf8]{inputenc}
\usepackage{tikz}
\usetikzlibrary {positioning}

\title{Lengths of Irreducible and Delicate Words}
\author{Benjamin Przybocki\thanks{Stanford University, {\tt benprz@stanford.edu}.}}
\date{August 2021}

\usepackage{graphicx}
\usepackage{amsmath,amsthm,amssymb,mathabx,mathrsfs}
\usepackage{fullpage}
\usepackage{mathtools}
\usepackage{hyperref}
\usepackage{comment}
\usepackage[english]{babel}
\usepackage{longtable}

\hypersetup{
colorlinks = true,
citecolor = blue,
linkcolor = black
}

\theoremstyle{definition}
\newtheorem{definition}{Definition}[section]

\theoremstyle{plain}
\newtheorem{theorem}[definition]{Theorem}

\newtheorem{question}[definition]{Question}

\newtheorem{lemma}[definition]{Lemma}

\theoremstyle{remark}
\newtheorem*{claim}{Claim}

\begin{document}

\maketitle

\begin{abstract}
    We study words that barely avoid repetitions, for several senses of ``barely''. A squarefree (respectively, overlap-free, cubefree) word is \emph{irreducible} if removing any one of its interior letters creates a square (respectively, overlap, cube). A squarefree (respectively, overlap-free, cubefree) word is \emph{delicate} if changing any one of its letters creates a square (respectively, overlap, cube). We classify the lengths of irreducible and delicate squarefree, overlap-free, and cubefree words over binary and ternary alphabets.
\end{abstract}

\section{Introduction}

In combinatorics on words, it's common to study repetitions in words. Three kinds of repetition are squares, overlaps, and cubes. A \emph{square} is a word of the form $XX$, where $X$ is a nonempty word; an example in English is ``hotshots''. An \emph{overlap} is a word of the form $xYxYx$, where $x$ is a letter and $Y$ is a possibly empty word; an example in English is ``alfalfa''. Finally, a \emph{cube} is a word of the form $XXX$, where $X$ is a nonempty word; an example in English is ``hahaha''.

A \emph{factor} of a word is a contiguous subword. For instance, every cube contains an overlap as a factor, and every overlap contains a square as a factor. We say that a word is \emph{squarefree} (respectively, \emph{overlap-free}, \emph{cubefree}) if none of its factors is a square (respectively, overlap, cube).

It's natural to ask over which alphabets there exist arbitrarily long squarefree, overlap-free, or cubefree words. It's easy to see that every binary word of length at least 4 contains a square. At the same time, over 100 years ago, Thue \cite{Thue1906, Thue1912} (see also \cite{Berstel1995}) proved that there are arbitrarily long overlap-free binary words and arbitrarily long squarefree ternary words. Thus, when studying squarefree words, we use a ternary alphabet, and when studying overlap-free and cubefree words, we use a binary alphabet, since these are the smallest alphabets for which the questions are interesting.

Recently, there has been interest in studying words that are only barely squarefree, overlap-free, or cubefree, for several senses of ``barely''. To this end, Grytczuk, Kordulewski, and Niewiadomski \cite{Grytczuk2020} introduced the notion of extremality. Namely, they defined an \emph{extremal} squarefree word to be a squarefree word such that inserting any letter from the alphabet into the word (possibly at the beginning or end) creates a square. The definition analogously extends to overlap-free and cubefree words. The same authors proved that there are infinitely many extremal squarefree ternary words. Mol and Rampersad \cite{MolRampersad} refined this result by determining for which lengths extremal squarefree ternary words exist; in particular, such words exist for all lengths at least 87. In the same vein, Mol, Rampersad, and Shallit \cite{MolRampersadShallit} determined for which lengths extremal overlap-free binary words exist. While there are infinitely many such words, they do not exist for all sufficiently large lengths. It remains unknown whether there are extremal cubefree binary words.

Harju \cite{Harju2021} introduced the notion of irreducibility as a dual to extremality. First, we say a letter in a word is \emph{interior} if it is not the first or last letter of the word. Then an \emph{irreducible} squarefree word is a squarefree word of length at least 3 such that removing any one of its interior letters creates a square. The definition again analogously extends to overlap-free and cubefree words. Notice that requiring the letter one removes to be interior is essential, because any squarefree, overlap-free, or cubefree word will remain so after removing the first or last letter. Harju proved that there is an irreducible squarefree ternary word of length $n$ if and only if $n \in \{3, 6, 8, 9, 10, 11\} \cup \{m \mid m \geq 13\}$. We solve the analogous problems for irreducible overlap-free binary words and irreducible cubefree binary words with the following theorems, which are proved in Section~\ref{sec-irr}.

\begin{theorem} \label{thm-irr-over}
    There is an irreducible overlap-free binary word of length $n$ if and only if $n \in \{6, 8, 9, 10\} \cup \{m \mid m \geq 12\}$.
\end{theorem}

\begin{theorem} \label{thm-irr-cube}
    There is an irreducible cubefree binary word of length $n$ if and only if $n \in \{10, 14, 18, 19, 20\} \cup \{m \mid m \geq 22\}$.
\end{theorem}

As a third sense in which a word can be barely squarefree, overlap-free, or cubefree, we introduce the notion of delicacy. A \emph{delicate} squarefree word is a nonempty squarefree word such that changing any one of its letters to another letter from the alphabet creates a square. The definition again analogously extends to overlap-free and cubefree words. We determine the possible lengths of delicate squarefree ternary words, delicate overlap-free binary words, and delicate cubefree binary words with the following theorems, which are proved in Section~\ref{sec-del}.

\begin{theorem} \label{thm-del-sq}
    There is a delicate squarefree ternary word of length $n$ if and only if $n \in \{5\} \cup \{m \mid m \geq 7\}$.
\end{theorem}

\begin{theorem} \label{thm-del-over}
    There is a delicate overlap-free binary word of length $n$ if and only if $n \in \{m \mid m \geq 7\}$.
\end{theorem}

\begin{theorem} \label{thm-del-cube}
    There is a delicate cubefree binary word of length $n$ if and only if $n \in \{20, 21, 22, 29, 33, 34, 35\} \cup \{m \mid m \geq 38\}$.
\end{theorem}

In Section~\ref{sec-ex-irr-del}, we exhibit an infinite family of overlap-free binary words that are simultaneously extremal, irreducible, and delicate.

\begin{theorem} \label{thm-ex-irr-del}
    There are infinitely many simultaneously extremal, irreducible, and delicate overlap-free binary words.
\end{theorem}

Finally, in Section~\ref{sec-k-del}, we conclude by introducing a natural generalization of delicacy and raising a question about it for further study.

\section{Irreducible Words} \label{sec-irr}

Let $\mu$ be the binary morphism defined by $\mu(0) = 01$ and $\mu(1) = 10$, and let $\mathbf{t} = \mu^\omega(0)$ be the Thue--Morse word, which is known to be overlap-free \cite{Thue1912}. For Theorem~\ref{thm-irr-over} we need the following lemma.

\begin{lemma}
    $010110 \mathbf{t}$ and $101001101001 \mathbf{t}$ are overlap-free.
\end{lemma}

\begin{proof}
    Thue \cite{Thue1912} proved that for any word $w$, $w$ is overlap-free if and only if $\mu(w)$ is. Since $010110 \mathbf{t} = \mu(001 \mathbf{t})$ and $101001101001 \mathbf{t} = \mu(110110 \mathbf{t})$, it suffices to know that $001 \mathbf{t}$ and $110110 \mathbf{t}$ are overlap-free, which was proven by Allouche, Currie, and Shallit \cite{Allouche1998}.
\end{proof}

We now prove Theorem~\ref{thm-irr-over}, which says that there is an irreducible overlap-free binary word of length $n$ if and only if $n \in \{6, 8, 9, 10\} \cup \{m \mid m \geq 12\}$.\footnote{Jeffrey Shallit has pointed out that this theorem can also be proven using the automated theorem prover Walnut.}

\begin{proof}[Proof of Theorem~\ref{thm-irr-over}]
    First, suppose $n \in \{6, 10\}$. In this case we use one of the following words:
    \[
    \begin{tabular}{ c | l }
        6 & 010010 \\
        10 & 0100101101.
    \end{tabular}
    \]

    Now suppose $n \in \{8, 9\} \cup \{m \mid m \geq 12\}$. Let $\mathbf{t}_k$ be the first $8k$ letters of $\mathbf{t}$. We claim that $\mathbf{t}_k$ is irreducible overlap-free. Observe that $T_0 = 01101001$ and $T_1 = 10010110$ are irreducible overlap-free. Since $\mathbf{t}_k$ is a concatenation of copies of $T_0$ and $T_1$, it suffices to verify that $T_0 T_0$, $T_0 T_1$, $T_1 T_0$, and $T_1 T_1$ are irreducible overlap-free. (We have incidentally proven that $\mathbf{t}$ is irreducible overlap-free.)
    
    We first prove the theorem when $n \not\equiv 7 \pmod 8$. We denote the length of a word $w$ by $|w|$. Let $k$ be the largest integer $k$ such that $|\mathbf{t}_k| \leq n$. If $n - |\mathbf{t}_k| = 0$, we're done. Otherwise, based on $n - |\mathbf{t}_k|$, we choose the word of the desired length from the following table.
    \[
    \begin{tabular}{ c | l }
        1 & $1 \mathbf{t}_k$ \\
        2 & $1001101001 \mathbf{t}_{k-1}$ \\
        3 & $01001101001 \mathbf{t}_{k-1}$ \\
        4 & $1001 \mathbf{t}_k$ \\
        5 & $01001 \mathbf{t}_k$ \\
        6 & $010110 \mathbf{t}_k$
    \end{tabular}
    \]
    
    The irreducibility of these words can be verified by removing the interior letters at the beginning one at a time and finding the overlaps in the resulting words. It follows from the lemma that these words are overlap-free.
    
    Next, we prove the result when $n \equiv 7 \pmod 8$ by giving a family of words that are irreducible overlap-free for $n \geq 39$ and $n \equiv 7 \pmod{16}$ and a family of words that are irreducible overlap-free for $n \geq 15$ and $n$ congruent to 15, 23, or 31 modulo 32. The first family is obtained by removing the first 14 letters of $\mathbf{t}$ and then taking prefixes, while the second is obtained by removing the first 15 letters of $\mathbf{t}$ and then taking prefixes. The words in these families are overlap-free since $\mathbf{t}$ is, so we only need to prove irreducibility.
    
    We prove irreducibility by induction on the length. The base cases can be verified. For the first family, we now show that appending the next 16 letters to a word in the family creates another irreducible overlap-free word. Notice that the next 16 letters must be of the form $z = X_1 T_i X_2$, where $X_1 \in \{001, 110\}$, $X_2 \in \{01101, 10010\}$, and $i \in \{0,1\}$. Not every combination of these is a possible value of $z$, however. First, $z \neq 001 T_0 01101$ and $z \neq 110 T_1 10010$ because they require $z$ to be from a factor of the form $T_0 T_0 T_0$ or $T_1 T_1 T_1$, respectively. Further, $z \neq 001 T_0 10010$ and $z \neq 110 T_1 01101$ because they require $z$ to be from a factor of the form $T_0 T_0 T_1$ or $T_1 T_1 T_0$, respectively, and such factors can only occur in $\mathbf{t}$ starting at an index congruent to 8 modulo 16. Thus, the only possibilities for $z$ are $001 T_1 01101$, $001 T_1 10010$, $110 T_0 01101$, and $110 T_0 10010$. In the first two cases, the 5 letters preceding them must be 01101, and these suffixes of length 21 are irreducible overlap-free. Similar reasoning holds for the latter two cases, in which case the 5 letters preceding them must be 10010.
    
    We now prove the induction hypothesis for the second family by showing that appending the next 32 letters to a word in the family creates another irreducible overlap-free word. The next 32 letters must be of the form $z = X_1 T_i T_j T_k X_2$, where $X_1 \in \{01, 10\}$, $X_2 \in \{011010, 100101\}$, and $i, j, k \in \{0,1\}$. Again, not every combination of these is a possible value of $z$. First, any combination that requires $z$ to be from a factor containing an overlap is impossible. Further, the combinations that require $z$ to be from a factor of the form $T_0 T_1 T_0 T_1 T_1$ or $T_1 T_0 T_1 T_0 T_0$ are impossible because such factors can only occur in $\mathbf{t}$ starting at an index congruent to 16 modulo 32. Finally, $\mathbf{t}$ contains no factors of the form $T_0 T_0 T_1 T_0 T_0$ or $T_1 T_1 T_0 T_1 T_1$. This leaves 10 possibilities for $z$. For each possibility, we can deduce the 6 letters preceding them (either 011010 or 100101), and the resulting suffixes of length 38 are irreducible overlap-free.
    
    For $n \not\in \{6, 8, 9, 10\} \cup \{m \mid m \geq 12\}$, a computer search shows there are no irreducible overlap-free binary words of length $n$.
\end{proof}

For Theorem~\ref{thm-irr-cube}, we need the following lemmas due to Richomme and Wlazinski \cite[Corollary~1]{Richomme2000} and Shur \cite[Proposition~2.1]{Shur2012}, respectively.

\begin{lemma} \label{lem-cubefree-morph}
    A binary morphism preserves cubefreeness if and only if the images of all cubefree binary words of length 7 are cubefree.
\end{lemma}

\begin{lemma} \label{lem-tm-sq}
    No prefix of $\mathbf{t}$ is a square.
\end{lemma}

We now prove Theorem~\ref{thm-irr-cube}, which says that there is an irreducible cubefree binary word of length $n$ if and only if $n \in \{10, 14, 18, 19, 20\} \cup \{m \mid m \geq 22\}$.

\begin{proof}[Proof of Theorem~\ref{thm-irr-cube}]
    First, suppose $n \in \{10, 14, 20, 24, 28\}$. In this case we use one of the following words:
    \[
    \begin{tabular}{ c | l }
        10 & 0100101101 \\ 
        14 & 01001011010010 \\  
        20 & 01001010011001010010 \\
        24 & 010010100110010100101101 \\
        28 & 0100101001100101001011010010.
    \end{tabular}
    \]
    
    Now suppose $n \in \{18, 19, 22, 23, 25, 26, 27\} \cup \{m \mid m \geq 29\}$. Consider the following morphisms:
    \begin{align*}
        \varphi_1(0) &= 01100100101101001011010010 \\
        \varphi_1(1) &= 0110101100110101100101001100101001 \\
        \varphi_2(0) &= 01001011010010110100100110 \\
        \varphi_2(1) &= 1001010011001010011010110011010110.
    \end{align*}
    Notice that $\varphi_2(0)$ and $\varphi_2(1)$ are the reversals of $\varphi_1(0)$ and $\varphi_1(1)$, respectively. The images of all cubefree binary words of length 7 under $\varphi_1$ and $\varphi_2$ are cubefree, so by Lemma~\ref{lem-cubefree-morph}, $\varphi_1$ and $\varphi_2$ preserve cubefreeness. Further, $\varphi_i(0)$, $\varphi_i(1)$, $\varphi_i(00)$, $\varphi_i(01)$, $\varphi_i(10)$, and $\varphi_i(11)$ are irreducible cubefree for $i \in \{1, 2\}$, so applying $\varphi_1$ or $\varphi_2$ to a prefix of a cubefree binary word results in an irreducible cubefree binary word. Let $w_1$ and $w_2$ be the images of $\mathbf{t}$ under $\varphi_1$ and $\varphi_2$, respectively, and let $w_{1,k}$ and $w_{2,k}$ be the images of the first $k$ letters of $\mathbf{t}$ under $\varphi_1$ and $\varphi_2$, respectively. Notice that $|w_{1,k}| = |w_{2,k}|$.
    
    Let $k$ be the largest integer $k$ such that $|w_{1,k}| \leq n$. If $n - |w_{1,k}| = 0$, we're done. Otherwise, based on $n - |w_{1,k}|$, we choose a word of the desired length from the following table.
    \begin{longtable}{ c | l }
        1 & $1 w_{1,k}$ \\ 
        2 & $0100101001100101001011001010 w_{1,k-1}$ \\  
        2 & $010010100110010100101101001011010010 w_{1,k-1}$ \\
        3 & $10110100101100101001100101001 w_{1,k-1}$ \\
        3 & $0100101101100100101100101001100101001 w_{1,k-1}$ \\
        4 & $0110 w_{2,k}$ \\
        5 & $01001 w_{1,k}$ \\
        6 & $10010100110010100101001100101001 w_{1,k-1}$ \\
        6 & $0100101100100101101001011001001011010010 w_{1,k-1}$ \\
        7 & $1001010 w_{1,k}$ \\
        8 & $01001010 w_{1,k}$ \\
        9 & $101101001 w_{1,k}$ \\
        10 & $010010100110010100101101001011010010 w_{1,k-1}$ \\
        10 & $01001011001001011010010110100101001100101001 w_{1,k-1}$ \\
        11 & $10010100110 w_{2,k}$ \\
        12 & $101101001010 w_{1,k}$ \\
        13 & $0100101101001 w_{1,k}$ \\
        14 & $01001011001010 w_{1,k}$ \\
        15 & $101101011001101 w_{1,k}$ \\
        16 & $0100101101001010 w_{1,k}$ \\
        17 & $01001010011001010 w_{1,k}$ \\
        18 & $100101001100101001 w_{1,k}$ \\
        19 & $0100101001100101001 w_{1,k}$ \\
        20 & $01001011011001001010 w_{1,k}$ \\
        21 & $100101001010011001010 w_{1,k}$ \\
        22 & $0100101101001011010010 w_{1,k}$ \\
        23 & $10110100101001100101001 w_{1,k}$ \\
        24 & $010010110100101101001010 w_{1,k}$ \\
        25 & $0100101100101001100101001 w_{1,k}$ \\
        26 & $01100100101101001011010010 w_{1,k}$ \\
        27 & $010010110100101001100101001 w_{1,k}$ \\
        28 & $0100101001100101001011001010 w_{1,k}$ \\
        29 & $10110100101100101001100101001 w_{1,k}$ \\
        30 & $100110110100101101001011011001 w_{1,k}$ \\
        31 & $0100101100100101101001011010010 w_{1,k}$ \\
        32 & $10010100110010100101001100101001 w_{1,k}$ \\
        33 & $010010100110010100101001100101001 w_{1,k}$
    \end{longtable}
    
    The irreducibility of these words can be verified by removing the interior letters at the beginning one at a time and finding the cubes in the resulting words. To prove that these words are cubefree, we need the following claim.
    
    \begin{claim}
        No prefix of $w_1$ or $w_2$ is an overlap.
    \end{claim}
    
    \begin{proof}
        We prove the stronger result that no prefix of $w_1$ or $w_2$ is a square. Suppose $w_1$ has a square prefix $XX$. Then $|X| \geq 8$. But in $w_1$, 01100100 occurs only at the beginning of $\varphi_1(0)$, so $XX$ must be the image of a square prefix of $\mathbf{t}$ under $\varphi_1$, which contradicts Lemma~\ref{lem-tm-sq}. Similarly, suppose $w_2$ has a square prefix $XX$. Then $|X| \geq 9$. But in $w_2$, 010010110 occurs only at the beginning of $\varphi_2(0)$, so we again have a contradiction.
    \end{proof}
    
    Now suppose one of the words in the above table contains a cube $XXX$. Then $XXX$ must start in the prefix (before $w_{1,k}$, $w_{1,k-1}$, or $w_{2,k}$), since $w_1$ and $w_2$ are cubefree. But one can verify that $X$ also can't be entirely within the prefix, so $X = PY$, where $P$ is a suffix of the prefix and $Y$ is a prefix of $w_1$ or $w_2$. Thus, $XXX = PYPYPY$, and $YPYPY$ is a prefix of $w_1$ or $w_2$, which contradicts the claim.
    
    For $n \not\in \{10, 14, 18, 19, 20\} \cup \{m \mid m \geq 22\}$, a computer search shows there are no irreducible cubefree binary words of length $n$.
\end{proof}

\section{Delicate Words} \label{sec-del}

Let $\tau$ be the ternary morphism defined by $\tau(0) = 012$, $\tau(1) = 02$, and $\tau(2) = 1$, and let $\mathbf{v} = \tau^\omega(0)$ be the ternary Thue--Morse word, which is known to be squarefree \cite{Istrail1977}. For Theorem~\ref{thm-del-sq}, we need the following lemma.

\begin{lemma} \label{lem-2v}
    $2\mathbf{v}$ is squarefree.
\end{lemma}

\begin{proof}
    Berstel \cite{Berstel1978} proved that $\mathbf{v}$ is equivalently characterized by letting the $i$-th letter be the number of 0s between the $i$-th and $(i+1)$-th 1s in $\mathbf{t}$. Thus, $2\mathbf{v}$ is defined by letting the $i$-th letter be the number of 0s between the $i$-th and $(i+1)$-th 1s in $10\mathbf{t}$. Since $10\mathbf{t}$ is overlap-free by a result of \cite{Allouche1998}, $2\mathbf{v}$ is squarefree.
\end{proof}

We also need the following lemma due to Crochemore \cite[Corollary~5]{Crochemore1982}.

\begin{lemma} \label{lem-squarefree-morph}
    A ternary morphism preserves squarefreeness if and only if the images of all squarefree ternary words of length 5 are squarefree.
\end{lemma}

We now prove Theorem~\ref{thm-del-sq}, which says that there is a delicate squarefree ternary word of length $n$ if and only if $n \in \{5\} \cup \{m \mid m \geq 7\}$.

\begin{proof}[Proof of Theorem~\ref{thm-del-sq}]
    Consider the following morphism:
    \begin{align*}
        \varphi(0) &= 01202120102 \\
        \varphi(1) &= 01210201021 \\
        \varphi(2) &= 01210212021.
    \end{align*}
    
    The images of all squarefree ternary words of length 5 under $\varphi$ are squarefree, so by Lemma~\ref{lem-squarefree-morph}, $\varphi$ preserves squarefreeness. Further, $\varphi(a)$ is delicate squarefree for $a \in \{0, 1, 2\}$, so applying $\varphi$ to a prefix of a squarefree ternary word results in a delicate squarefree ternary word. Let $w$ be the image of $\mathbf{v}$ under $\varphi$, and let $w_k$ be the image of the first $k$ letters of $\mathbf{v}$ under $\varphi$.
    
    Let $k$ be the largest integer $k$ such that $|w_k| \leq n$. If $n - |w_k| = 0$, we're done. Otherwise, based on $n - |w_k|$, we choose a word of the desired length from the following table.
    \[
    \begin{tabular}{ c | l }
        1 & $010210120102 w_{k-1}$ \\
        2 & $02 w_k$ \\
        3 & $102 w_k$ \\
        4 & $0121 w_k$ \\
        5 & $12021 w_k$ \\
        6 & $012102 w_k$ \\
        7 & $0212021 w_k$ \\
        8 & $02120121 w_k$ \\
        9 & $021012102 w_k$ \\
        10 & $1202120121 w_k$
    \end{tabular}
    \]
    
    The delicacy of these words can be verified by changing the letters at the beginning one at a time and finding the squares in the resulting words. We prove that these words are squarefree with the following claims.
    
    \begin{claim}
        $010210120102 w$ and $021012102 w$ are squarefree.
    \end{claim}
    
    \begin{proof}
        Suppose one of these words contains a square $XX$. Then $XX$ must start in the prefix (before $w$), since $w$ is squarefree. Further, one can verify that $X$ must contain the factor 20120, which is a contradiction, since $w$ does not contain this factor.
    \end{proof}
    
    \begin{claim}
        $0212021 w$ is squarefree.
    \end{claim}
    
    \begin{proof}
        This follows from Lemma~\ref{lem-2v}.
    \end{proof}
    
    \begin{claim}
        $1202120121 w$ is squarefree.
    \end{claim}
    
    \begin{proof}
        Suppose this word contains a square $XX$. Since $w$ does not contain 12101 as a factor, $XX$ is contained in $21 w$. But $21 w$ is squarefree by Lemma~\ref{lem-2v}, which is a contradiction.
    \end{proof}
    
    For $n \not\in \{5\} \cup \{m \mid m \geq 7\}$, a computer search shows there are no delicate squarefree ternary words of length $n$.
\end{proof}

We now prove Theorem~\ref{thm-del-over}, which says that there is a delicate overlap-free binary word of length $n$ if and only if $n \in \{m \mid m \geq 7\}$.\footnote{Jeffrey Shallit has pointed out that this theorem can also be proven using the automated theorem prover Walnut.}

\begin{proof}[Proof of Theorem~\ref{thm-del-over}]
    If $n = 9$, we use the word 001011001. Now suppose $n \in \{7, 8\} \cup \{m \mid m \geq 10\}$. We use a construction for each residue class modulo 8. Based on the residue class, we use a prefix of $\mathbf{t}$ after removing the first $k$ letters, where $k$ is given by the following table.
    \[
    \begin{tabular}{ c | l }
        0 & $k = 0$ \\
        1 & $k = 7$ \\
        2 & $k = 6$ \\
        3 & $k = 13$ \\
        4 & $k = 12$ \\
        5 & $k = 3$ \\
        6 & $k = 10$ \\
        7 & $k = 1$
    \end{tabular}
    \]
    
    As factors of $\mathbf{t}$, these words are overlap-free. Their delicacy is shown by induction. The base cases can be verified. The induction hypothesis is simple because both $T_0$ and $T_1$ are delicate overlap-free.
    
    For $n \not\in \{m \mid m \geq 7\}$, a computer search shows there are no delicate overlap-free binary words of length $n$.
\end{proof}

We now prove Theorem~\ref{thm-del-cube}, which says that there is a delicate cubefree binary word of length $n$ if and only if $n \in \{20, 21, 22, 29, 33, 34, 35\} \cup \{m \mid m \geq 38\}$.

\begin{proof}[Proof of Theorem~\ref{thm-del-cube}]
    First, suppose $n \in \{20, 33\}$. In this case we use one of the following words:
    \[
    \begin{tabular}{ c | l }
        20 & 00101001101001101011 \\
        33 & 001010011010011010110010110010100.
    \end{tabular}
    \]
    
    Now suppose $n \in \{21, 22, 29, 34, 35\} \cup \{m \mid m \geq 38\}$. Consider the following morphism:
    \begin{align*}
        \varphi(0) &= 0110101100101100101001 \\
        \varphi(1) &= 1001010011010011010110.
    \end{align*}
    
    The images of all cubefree binary words of length 7 under $\varphi$ are cubefree, so by Lemma~\ref{lem-cubefree-morph}, $\varphi$ preserves cubefreeness. Further, both $\varphi(0)$ and $\varphi(1)$ are delicate cubefree, so applying $\varphi$ to a prefix of a cubefree binary word results in a delicate cubefree binary word. Let $w$ be the image of $\mathbf{t}$ under $\varphi$, and let $w_k$ be the image of the first $k$ letters of $\mathbf{t}$ under $\varphi$.
    
    Let $k$ be the largest integer $k$ such that $|w_k| \leq n$. If $n - |w_k| = 0$, we're done. Otherwise, based on $n - |w_k|$, we choose a word of the desired length from the following table.
    \[
    \begin{tabular}{ c | l }
        1 & $00101001101001101011001 w_{k-1}$ \\
        2 & $011001001100110110011001 w_{k-1}$ \\
        3 & $0010100110100110101101001 w_{k-1}$ \\
        4 & $00101001101001101011001010 w_{k-1}$ \\
        5 & $001010011010011010110010110 w_{k-1}$ \\
        6 & $0010100110100110101100101001 w_{k-1}$ \\
        7 & $01100100110011011001100100110 w_{k-1}$ \\
        8 & $001010011010011010110100101001 w_{k-1}$ \\
        9 & $0010100110100110101100101001010 w_{k-1}$ \\
        10 & $00101001101001101011001010011010 w_{k-1}$ \\
        11 & $001010011010011010110010110011001 w_{k-1}$ \\
        12 & $001010011010 w_k$ \\
        13 & $1001010011010 w_k$ \\
        14 & $001010011010011010110010100101001101 w_{k-1}$ \\
        15 & $0010100110100110101100100110011011001 w_{k-1}$ \\
        16 & $01100100110011011001100100110011011001 w_{k-1}$ \\
        17 & $00101001101001101 w_k$ \\
        18 & $100101001101001101 w_k$ \\
        19 & $1101011001011001010 w_k$ \\
        20 & $01101011001011001010 w_k$ \\
        21 & $001010011010011010110 w_k$
    \end{tabular}
    \]
    
    The delicacy of these words can be verified by changing the letters at the beginning one at a time and finding the cubes in the resulting words. To prove that these words are cubefree, we need the following claim.
    
    \begin{claim}
        No prefix of $w$ is an overlap.
    \end{claim}
    
    \begin{proof}
        We prove the stronger result that no prefix of $w$ is a square. Suppose $w$ has a square prefix $XX$. Then $|X| \geq 12$. But in $w$, 011010110010 occurs only at the beginning of $\varphi(0)$, so $XX$ must be the image of a square prefix of $\mathbf{t}$ under $\varphi$, which contradicts Lemma~\ref{lem-tm-sq}.
    \end{proof}
    
    Now suppose one of the words in the above table contains a cube $XXX$. Then $XXX$ must start in the prefix (before $w_k$ or $w_{k-1}$), since $w$ is cubefree. But one can verify that $X$ also can't be entirely within the prefix, so $X = PY$, where $P$ is a suffix of the prefix and $Y$ is a prefix of $w$. Thus, $XXX = PYPYPY$, and $YPYPY$ is a prefix of $w$, which contradicts the claim.
    
    For $n \not\in \{20, 21, 22, 29, 33, 34, 35\} \cup \{m \mid m \geq 38\}$, a computer search shows there are no delicate cubefree binary words of length $n$.
\end{proof}

\section{Extremal, Irreducible, and Delicate Overlap-free Binary Words} \label{sec-ex-irr-del}

For Theorem~\ref{thm-ex-irr-del}, we need the following lemma due to Mol, Rampersad, and Shallit \cite[Lemma~6]{MolRampersadShallit}.

\begin{lemma}
    If $w = w^\prime w^{\prime\prime}$ is an overlap-free binary word with $|w^\prime|, |w^{\prime\prime}| \geq 5$, then every extension $w^\prime a w^{\prime\prime}$ for $a \in \{0, 1\}$ contains an overlap.
\end{lemma}

We now prove Theorem~\ref{thm-ex-irr-del}, which says that there are infinitely many simultaneously extremal, irreducible, and delicate overlap-free binary words.

\begin{proof}[Proof of Theorem~\ref{thm-ex-irr-del}]
    We define the following infinite family of words:
    \begin{align*}
        w_0 &= 01100110100110010110011010011001 \\
        w_{n+1} &= \mu(w_n).
    \end{align*}
    Since $w_0$ is overlap-free and $w$ is overlap-free if and only if $\mu(w)$ is for any word $w$ \cite{Thue1912}, we conclude that $w_i$ is overlap-free for all $i$. We complete the proof with the following claims.
    
    \begin{claim}
        $w_i$ is extremal overlap-free for all $i$.\footnote{The proof of this claim is similar to the proof of Lemma~12 in \cite{MolRampersadShallit}.}
    \end{claim}
    
    \begin{proof}
        The extremality of $w_0$ can be verified, so suppose $i \geq 1$. By the lemma, we only need to verify that every extension $w^\prime a w^{\prime\prime}$ for $a \in \{0, 1\}$ with $|w^\prime| \leq 4$ or $|w^{\prime\prime}| \leq 4$ contains an overlap. In fact, we only need to consider the case where $|w^\prime| \leq 4$, because when $i$ is even, $w_i$ is the complement of its reversal, and when $i$ is odd, $w_i$ is its own reversal. We consider two cases.
        
        \textbf{Case I:} $|w^\prime| = 0$. Suppose $a = 0$. If $i$ is even, the first quarter of $w_i$ is of the form $X0X0$, so we have an overlap. If $i$ is odd, $w_i$ is of the form $X0X0$, so we have an overlap.  Suppose $a = 1$. If $i$ is even, $w_i$ is of the form $X1X1$, so we have an overlap. If $i$ is odd, the first quarter of $w_i$ is of the form $X1X1$, so we have an overlap.
        
        \textbf{Case II:} $1 \leq |w^\prime| \leq 4$. The first 8 letters of $w_i$ are 01101001. If $|w^\prime| = 1$ and $a = 0$, this is equivalent to $|w^\prime| = 0$ and $a = 0$, so the result follows from Case I. Otherwise, inserting 0 or 1 into 01101001 creates an overlap.
    \end{proof}
    
    \begin{claim}
        $w_i$ is irreducible overlap-free for all $i$.
    \end{claim}
    
    \begin{proof}
        We again verify the result for $w_0$ and then suppose $i \geq 1$. Since $w_i$ is a concatenation of copies of $T_0$ and $T_1$, it suffices to verify that $T_0 T_0$, $T_0 T_1$, $T_1 T_0$, and $T_1 T_1$ are irreducible overlap-free.
    \end{proof}
    
    \begin{claim}
        $w_i$ is delicate overlap-free for all $i$.
    \end{claim}
    
    \begin{proof}
        We again verify the result for $w_0$ and then suppose $i \geq 1$. Since $T_0$ and $T_1$ are delicate overlap-free, the claim follows.
    \end{proof}
    
    This concludes the proof of the theorem.
\end{proof}

\section{Generalizing Delicacy} \label{sec-k-del}

We conclude by introducing a natural generalization of delicacy and raising a question about it for further study. A \emph{$k$-delicate} squarefree word is a nonempty squarefree word such that changing between 1 and $k$ of its letters to other letters from the alphabet creates a square. The definition again analogously extends to overlap-free and cubefree words.

\begin{question}
    Are there finite $k$-delicate squarefree (respectively, overlap-free, cubefree) ternary (respectively, binary) words for all $k$?
\end{question}

\section{Acknowledgments}
This research was conducted at the University of Minnesota Duluth research program run by Joe Gallian and supported by NSA Grant H98230-20-1-0009 and NSF-DMS Grant 1949884. Thanks to Amanda Burcroff, Joe Gallian, Swapnil Garg, Noah Kravitz, and Jeffrey Shallit for helpful comments.

\newpage
\bibliographystyle{plain}
\bibliography{bib}

\end{document}